%% file: WassersteinEigendistance.tex
\documentclass[a4paper,10pt]{scrartcl}

\usepackage[fleqn]{amsmath}
\usepackage{amssymb}
\usepackage{amsthm}
\usepackage{stmaryrd}
\usepackage{bbm}
\usepackage{hyperref}
\usepackage{xcolor}
\usepackage{mathtools}
\usepackage{nicefrac}

\mathtoolsset{showonlyrefs}

\include{Include}

\include{Theorems}

\newcommand{\hP}{\widehat{\bP}}
\newcommand{\hE}{\widehat{\bE}}

\newcommand{\orho}{\overline{\rho}}
\newcommand{\hPP}{\widehat{P}}
\newcommand{\diag}{\operatorname{diag}}

\title{Constant curvature metrics for Markov chains}
\author{F. V{\"o}llering\footnote{University of Bath, f.m.vollering@bath.ac.uk}}

\begin{document}

\maketitle

\begin{abstract}
We consider metrics which are preserved under a  $p$-Wasserstein transport map, up to a possible contraction. In the case $p=1$ this corresponds to a metric which is uniformly curved in the sense of coarse Ricci curvature. We investigate the existence of such metrics in the more general sense of pseudo-metrics, where the distance between distinct points is allowed to be 0, and show the existence for general Markov chains on compact Polish spaces. Further we discuss a notion of algebraic reducibility and its relation to the existence of multiple true pseudo-metrics with constant curvature. Conversely, when the Markov chain is irreducible and the state space finite we obtain effective uniqueness of a metric with uniform curvature up to scalar multiplication and taking the $p$th root, making this a natural intrinsic distance of the Markov chain. An application is given in the form of concentration inequalities for the Markov chain.
\end{abstract}

\section{Introduction}
In \cite{OLLIVIER:09} the contraction rate of a Markov transition kernel is interpreted as coarse Ricci curvature, based on the metric structure of the underlying space and the Wasserstein distances between one-step probability distributions. Note that the curvature is no longer an intrinsic property of the underlying space, but is a consequence of how a Markov chain acts on this space. Consequently, for a given metric space different Markov transition kernels induce different curvatures, which are in a sense adapted to describe the evolution of the corresponding Markov chain. In practice this implicitly uses that the underlying metric structure of the space is well-chosen to begin with. This leads to the question what a natural metric would be for a given Markov chain to go in hand with coarse Ricci curvature.

In this paper we explore this question by looking for metrics in which the space becomes uniformly curved. We call such a metric a Wasserstein eigendistance, and the concept and existence theorems are the subject of Section \ref{section:Wasserstein}. Such a (pseudo-)metric is an intrinsic property of the Markov chain. It turns out that uniqueness of these eigendistances is related to a form of algebraic irreducibility, which is discussed in Section \ref{section:irreducibility}. In Section \ref{section:examples} we look at various examples, and in Section \ref{section:concentration} we provide as an application of Wasserstein eigendistances concentration estimates for functions which are Lipschitz with respect to an eigendistance. More discussion and  open problems are presented in Section \ref{section:discussion}. 
The last sections are then dedicated to proofs.
We do not focus on general consequences of positive coarse Ricci curvature and instead refer to \cite{OLLIVIER:09}.

\section{Wasserstein-Eigendistances}\label{section:Wasserstein}
Let $E$ be a compact Polish space, and $(X_t)_{t\in\bN}$ a Markov chain on $E$ with transition kernels $(P^x)_{x\in E}$, with corresponding law $\bP_x$ and expectation $\bE_x$. We will also write $P$ for the corresponding transition operator on functions, so that we use interchangeably 
\[ \bE_xf(X_1) = \int f(z)\,P^x(dz) = \int f \,dP^x = P^x(f) = (Pf)(x) \]
and $\bE_xf(X_n)=P^nf(x)$, $f:E\to\bR$, $n\in\bN$.

A metric is a function $\rho:E\times E\to[0,\infty)$ satisfying $\rho(x,y)=\rho(y,x)$, $\rho(x,y)\leq \rho(x,z)+\rho(z,y)$ and $\rho(x,y)=0$ iff $x=y$. There are natural generalizations which relax these assumptions slightly. An extended metric allows $\rho$ to take the value $\infty$, and a pseudo-metric allows $\rho(x,y)=0$ for $x\neq y$. Denote by $\cM_{[0,\infty]}$ the set of all measurable extended pseudo-metrics.

Denote by $D=\{(x,x):x\in E\}$ the diagonal set of $E\times E$. To work with bounded subsets of $\cM_{[0,\infty]}$, consider two functions $f_1,f_2:E\times E\setminus D\to [0,\infty]$ and denote by $\cM_{[f_1,f_2]}$ ($\cM_{(f_1,f_2]}$,...) be the subset of all $\rho\in \cM_{[0,1]}$ satisfying $f_1(x,y)\leq \rho(x,y) \leq f_2(x,y)$ ($f_1(x,y)< \rho(x,y) \leq f_2(x,y)$, ...) for all $x,y \in E\times E\setminus D$.
In particular, $\cM_{(0,\infty)}$ is the set of all proper metrics. We say $\rho\in \cM_{[0,\infty]}$ is non-degenerate if there exist $x,y\in E$ with $\rho(x,y)\in (0,\infty)$ and say $\rho$ is a true pseudo-metric if it is non-degenerate and $\rho(x,y)=0$ for some $x\neq y$.

Let $\cP=\cP(E)$ be the set of all probability measures on $E$. For $\mu,\nu\in \cP(E)$, let $\cP_{\mu,\nu}=\cP_{\mu,\nu}(E^2)$ be the set of all couplings of $\mu$ and $\nu$, that is probability measures on $E\times E$ with marginal $\mu$ and $\nu$. In the case $\mu=P^x$ and $\nu=P^y$ we simply write $\cP_{x,y}$.
We equip $\cP$ and $\cP_{\mu,\nu}$ with the total variation distance.

The $p$-Wasserstein distance of $\mu, \nu\in\cP$, $p\in[1,\infty)$, with respect to an extended pseudo-metric $\rho\in\cM_{[0,\infty]} $ is given by
\begin{align}\label{eq:Wasserstein}
\cW_{p,\rho}(\mu,\nu) = \inf_{\pi\in\cP_{\mu,\nu}}\left(\int \rho(X,Y)^p\;d\pi(X,Y)\right)^\frac1p.
\end{align}
Typically it is defined only for proper metrics and under some first moment assumption to make the integral finite. But as we work with extended pseudo-metrics we can do without any first moment assumption. The Wasserstein distance is usually interpreted for a fixed metric $\rho$. Here we instead focus how it acts as a map on $\cM$. Therefore we write
\begin{align}\label{eq:W}
W_p:\cM_{[0,\infty]}\to\cM_{[0,\infty]}, \qquad W_p(\rho)(x,y):= \cW_{p,\rho}(P^x,P^y).
\end{align}
A basic property of $W_p$ is monotonicity, i.e., for $\rho_1\leq \rho_2$ we have $W_p(\rho_1)\leq W_p(\rho_2)$, where $\leq$ is the usual partial order of real-valued functions. We also have $W_p(\rho)\leq W_{p'}(\rho)$ for $p\leq p'$. If $\rho\in \cM_{[0,1]}$, then we also have $W_{p'}(\rho)\leq W_p(\rho)^{\frac{p}{p'}}$.

\begin{assumption}\label{as:regularity}
\begin{enumerate}
\item (continuity) The map $x\mapsto P^x$ is continuous in $\cP$.
\item (non-degeneracy) 
If $(\rho_n)_{n\in\bN}\subset \cM_{[0,1]}$ is a sequence with $\lim_{n\to\infty}W_p(\rho_n)=0$, then $\lim_{n\to\infty}\rho_n =0$.
\end{enumerate}
\end{assumption}
\begin{remark}
The non-degeneracy condition does not depend on the value of $p$, since the metrics are bounded and hence $\lim_{n\to\infty}W_p(\rho_n)=0$ iff $\lim_{n\to\infty}W_1(\rho_n)=0$.
\end{remark}
\begin{lemma}
Assume $E$ is discrete and $\inf_{x\in E}P^x(x)>\frac12$. Then Assumption \ref{as:regularity} is satisfied.
\end{lemma}
\begin{proof}
Part a) follows directly from the fact that $E$ is discrete. For part $b)$, let $x,y\in E$, $x\neq y$ and note that for any coupling $\pi\in \cP_{x,y}$ the probability to be stationary is bounded: There exists $\epsilon>0$ independent of $x,y$ so that
\[\pi(x,y)\geq 1-P^x(E\setminus x)-P^y(E\setminus y)\geq \epsilon> 0. \]
It follows that $W_p(\rho)\geq \epsilon^\frac1p\rho$.
\end{proof}

\begin{definition}
We say that an extended pseudo-metric $\rho\in\cM_{[0,\infty]} $ is a $p$-Wasserstein eigendistance with (uniform) curvature $\kappa\in[0,1]$ if 
\[ W_p(\rho)=(1-\kappa) \rho. \]
For $\cS \in\{(0,\infty),[0,\infty], [0,1]\}$ and $\kappa\in[0,1)$ we write
\begin{align}
E_p(\kappa)\cM_\cS := \left\{ \rho\in \cM_\cS \text{ non-degenerate} : W_p(\rho)= (1-\kappa) \rho \right\}
\end{align}
for the corresponding set of all such $\rho$. For $\Lambda\subset[0,1)$ we denote by $E_p(\Lambda) \cM_\cS := \bigcup_{\kappa\in \Lambda} E_p(\kappa)\cS$ the set of all $p$-Wasserstein eigendistances with uniform curvature in $\Lambda$.
\end{definition}
The naming suggests a relation to eigenfunctions and eigenvalues, with $\rho$ being an ``eigenfunction'' to ``eigenvalue'' $1-\kappa$. However $W_p$ is not linear, making the connection not straight forward. The following theorem provides mathematical justification beyond analogy.
\begin{definition}\label{def:coupling-operator}
We call a Markov transition operator $\hPP$ on $E\times E$ a coupling operator of $P$ if $\hPP^{x,y}\in\cP_{x,y}$ for all $x,y\in E$ and $\supp(\hPP^{x,x})\subset D=\{(x,x):x\in E\}$. 

We typically write $\hP_{x,y}$ for the law of Markov chain $(X_t,Y_t)_{t\in\bN}$ generated by $\hPP$ started in $(x,y)$, and $\hE_{x,y}$ the corresponding expectation, so that $(X_t)_{t\in\bN}$ has law $\bP_x$ and $(Y_t)_{t\in\bN}$ has law $\bP_y$.
\end{definition}
\begin{theorem}\label{thm:markov-coupling}
Assume Assumption \ref{as:regularity} and let $\rho\in E_p(\kappa) \cM_{[0,1]}$. Then there exists a coupling operator $\hPP$ of $P$ 
with  $\hPP^{xy}(\rho^p)=W_p(\rho)^p(x,y)=(1-\kappa)^p \rho^p(x,y)$. In particular, $\rho^p$ is an eigenfunction of $\hPP$ to eigenvalue $(1-\kappa)^p$ and the corresponding Markov chain $(X_t,Y_t)_{t\in\bN}$ satisfies $\hE_{x,y}\rho^p(X_t, Y_t) = (1-\kappa)^{pt}\rho^p(x,y)$.
\end{theorem}
In other words, the eigendistance $\rho$ is an eigenfunction to an appropriate coupling operator realizing the Wasserstein distance.
 Calling $\kappa$ the curvature instead of the eigenvalue is motivated by its geometric function. In \cite{OLLIVIER:09} the concept of a coarse Ricci curvature is introduced. For a given Markov chain and a given distance $\rho$, the coarse Ricci curvature between $x$ aned $y$ is defined by
 \begin{align}
 \kappa(x,y) = 1-\frac{W_1(\rho)(x,y)}{\rho(x,y)}.
 \end{align}
 If $\rho\in E_1(\kappa) M_{[0,1]}$, then $\kappa(x,y)=\kappa$ for all $x,y$, making $\rho$ uniformly curved with curvature $\kappa$.

Just based on the definition it is not clear that Wasserstein eigendistances exist, which leads us to the first existence theorem.
\begin{theorem}[Existence of $p$-Wasserstein eigendistances]\label{thm:eigen-distance}
Assume Assumption \ref{as:regularity}. Then, for any $p\in[1,\infty)$, $E_p([0,1))\cM_{[0,1]}$ is non-empty and its elements are continuous. 

More generally, let $\orho\in \cM_{[0,1]}$ be non-degenerate satisfying $W_p(\orho)\leq \orho$ and 
\begin{align}\label{eq:orho-inf}
\inf_{x,y: \orho(x,y)>0}\orho(x,y)>0.
\end{align}
Then 
$E_p([0,1))\cM_{[0,\orho]}\neq\emptyset$, and any element of $E_p([0,1))\cM_{[0,\orho]}$ is continuous.
\end{theorem}
After establishing existence there are two natural questions. First, is there a unique eigendistance(up to scalar multiplication)? Note that uniqueness can only hold up to multiplication by positive constants $\lambda>0$, since $W_p(\lambda \rho)=\lambda W_p(\rho)$. And second, is an eigendistance a proper metric? In general the answer is no to both questions, as the following proposition shows:
\begin{proposition}\label{prop:tensor}
Let $X_t$ and $Y_t$ be two independent Markov chains, and $\rho_X$, $\rho_Y$ be $p$-Wasserstein eigendistances with the curvature $\kappa_X$, $\kappa_Y$ to $X_t$ and $Y_t$ respectively. 

Then $\rho((x,y),(u,v))=\rho_X(x,u)$ ($=\rho_Y(y,v)$) is a $p$-Wasserstein eigendistance to the product chain $(X_t,Y_t)$ with curvature $\kappa_X$ ($\kappa_Y$).

If $\kappa_X=\kappa_Y$, then 
for any $a,b\geq 0$
\[
\rho((x,y),(u,v))=\left(a\rho_X^p(x,u)+b\rho_Y^p(y,v)\right)^\frac1p
\] 
is a $p$-Wasserstein eigendistance to the product chain $(X_t,Y_t)$ with curvature $\kappa_X=\kappa_Y$.
\end{proposition}
While this proposition shows that even if $X_t$ and $Y_t$ have unique Wasserstein eigendistances which are proper metrics, the product chain does not. But a product chain is special, with a lot of independence in its structure, and in Section \ref{section:irreducibility} we will explore this further and see that a Markov chain of product form is a special case of algebraic reducibility related to the existence of proper pseudo-metric eigendistances. 

The proof of Theorem \ref{thm:eigen-distance} uses a fixed point argument it is ill-suited to actually find Wasserstein eigendistances. In the remainder of this section we have two results which give some control on the shape of Wasserstein eigenfunctions, and which in fact do not assume Assumption \ref{as:regularity}. The first looks at the case of $0$ curvature, where we can identify a special eigendistance.
\begin{theorem}\label{thm:maximal}
Assume $E_p(0)\cM_{[0,1]}\neq \emptyset$. Then there exists a maximal element $\rho_*\in E_p(0)\cM_{[0,1]}$, and $\lim_{n\to\infty}W_p^n(\ind_{x\neq y})=\rho_*$.
\end{theorem}
The second result assumes that we know a suitable eigenfunction of $P$.
\begin{theorem}\label{thm:eigenfunction}
Assume Assumption \ref{as:regularity}. Assume $h$ is a non-negative eigenfunction to $P$ with eigenvalue $\lambda>0$, i.e., $Ph=\lambda h$. Then there exists a $p$-Wasserstein eigendistance $\rho$ with curvature $1-\lambda^\frac1p$ satisfying
\begin{align}
\abs{h(x)-h(y)}^\frac1p \leq \rho(x,y) \leq \ind_{x\neq y}(h(x) + h(y))^\frac1p.
\end{align}
\end{theorem}

\section{Algebraic irreducibility} \label{section:irreducibility}
By Theorem \ref{thm:eigen-distance} we know that there exist Wasserstein eigendistances. These share many properties with eigenfunctions and eigenvalues, and it is tempting to think of Theorem \ref{thm:eigen-distance} as a non-linear analogue of the Perron-Frobenius theorem, which suggests that the Markov chain needs to be irreducible to allow for a unique (up to multiplication) eigendistance which is positive. As we will see, the classic irreducibility of Markov chains is not the right condition, we need a different type of irreducibility. 
In a first step we see that having a true pseudo-metric as eigendistance implies additional structure for the Markov chain.
\begin{proposition}\label{prop:markov-galaxies}
Let $\rho\in E_p([0,1))\cM_{[0,1]}$ be continuous. Define the equivalence relation $\sim$ via $x\sim y$ iff $\rho(x,y)=0$, and denote by $[x]\in E/\sim$ the equivalence class of $x\in E$. Then
$([X_t])_{t\in \bN}$ is a Markov chain. 
\end{proposition}
If $\rho$ is a true pseudo-metric, then that means that $(X_t)_{t\in\bN}$ has an autonomous subchain in the form of $([X_t])_{t\in \bN}$, and the map $x\mapsto [x]$ preserves the Markovian structure of the process. In general the image of a Markov chain under some map $\phi$ is of course not a Markov chain. Here is a simple example for an autonomous subchain:
\begin{example}\label{ex:product-chain}
Let $(X_t)_{t\in \bN}$ and $(Y_t)_{t\in \bN}$ be two independent Markov chains. Then $(X_t,Y_t)_{t\geq 0}$ is a Markov chain, and the projection onto the first or second coordinate is again Markov. 
\end{example}
We formalize the concept of maps preserving the Markovian structure in the following definition:
\begin{definition}
Let $P$ be a Markov transition operator on $E$. We say a continuous map $\phi:E\to E'$ is a $P$-homomorphism if $Q^{\phi(x)}:= P^x\circ \phi^{-1}$, $x\in E$, is a well-defined Markov transition operator on a Polish space $E'$.
\end{definition}
With the concept of structure preserving homomorphisms it is natural to call $P$ algebraically irreducible when there are only trivial homomorphisms:
\begin{definition}
We say $P$ is algebraically irreducible if any $P$-homomorphism $\phi$ is either constant or bijective as a map from $E$ to $\phi(E)$, and call $P$ algebraically reducible otherwise.
\end{definition}
Example \ref{ex:product-chain} shows that algebraic irreducibility is distinct from the classic irreducibility of Markov chains, and Proposition \ref{prop:markov-galaxies} shows that if $P$ has an eigendistance which is a true pseudo-metric, then $P$ is algebraically reducible.
Theorem \ref{thm:markov-coupling} states that to any eigendistance $\rho$ there is an associated(in general not unique) coupling realizing $\rho$. We will investigate the connection between $\rho$ and its coupling. Since $\rho$ is symmetric we restrict ourself to symmetric couplings.
\begin{definition}\label{def:symmetric-coupling}
Consider the equivalence relation $\sim_{sym}$ on $E\times E$
given by $(x,y)\sim_{sym}(y,x)$, and denote the elements of $(E\times E)_{sym}=E\times E/\sim_{sym}$ by $\{x,y\}$.
An event $A\subset E\times E$ is symmetric if $A=A^T$, where $A^T=\{(x,y): (y,x)\in A\}$. Symmetric events in $E\times E$ naturally map one-to-one to events in $(E\times E)_{sym}$.

We say a coupling operator $\hPP$ is symmetric if for any event $A\subset E\times E$ and $x,y\in E$ we have $\hPP^{x,y}(A) = \hPP^{y,x}(A^T)$. As a consequence, $\hPP^{\{x,y\}}_{sym}(A)= \hPP^{x,y}(\{(x,y): \{x,y\}\in A\})$, $A\subset (E\times E)_{sym}$, is a well-defined Markov transition operator on $(E\times E)_{sym}$.
\end{definition}
\begin{remark}\label{remark:symmetric}
A non-symmetric coupling operator $\hPP$ can be symmetrized by working with $\widehat{Q}^{x,y}(A) = \frac12(\hPP^{x,y}(A)+\hPP^{y,x}(A^T))$. In particular the coupling operator in Theorem \ref{thm:markov-coupling} can be assumed to be symmetric.
\end{remark}

\begin{definition}
We say a symmetric coupling operator $\hPP$ is irreducible if $\hPP_{sym}$ is irreducible outside the diagonal. That is, for any $x,y,x',y'\in E$, $x\neq y$, $x'\neq y'$, and any open neighborhood $U$ of $\{x',y'\}$ in $(E\times E)_{sym}$ there exists $n\in \bN$ so that $(\hPP^{\{x,y\}}_{sym})^n(U)>0$.
\end{definition}
Note that if $E$ is countable the above definition is the classic notion of irreducibility except for the diagonal. But since for coupling operators the diagonal has to be absorbing this is the natural extension of classical Markov chain irreducibility to coupling operators. Beside the intrinsic symmetry of metrics there is an additional reason why we work with symmetric couplings. On $E\times E$ it is possible that the diagonal naturally splits the state space into two sets, and correspondingly there are two non-diagonal irreducibility classes for a coupling $\hPP$. Working with $\hPP_{sym}$ avoids this issue.

We can now state the first main result of this section:
\begin{theorem}\label{thm:irreducibility1} Assume Assumption \ref{as:regularity}. 
The following statements are equivalent:
\begin{enumerate}
\item $P$ is algebraically irreducible.
\item For any $p\in[1,\infty)$, every element of $E_p([0,1))\cM_{[0,1]}$ of $P$ is a proper metric.
\item Every symmetric coupling operator $\hPP$ of $P$ is irreducible.
\end{enumerate}
\end{theorem}

%

We conclude this section with a result on the uniqueness of the $p$-Wasserstein eigendistance when $E$ is finite.
\begin{theorem}[Uniqueness of Wasserstein eigendistances]\label{thm:uniqueness}
Assume $E$ is finite and $P$ is algebraically irreducible. Then there exists a unique metric $\rho_*\in\cM_{(0,1]}$ with $\operatorname{diam}_{\rho_*}(E)=1$ and a $\kappa_*\in[0,1)$ so that for any $p\in[1,\infty)$
\begin{align*}
E_p([0,1))\cM_{[0,\infty)} = E_p(1-(1-\kappa_*)^\frac1p)\cM_{(0,\infty)} = \{ \lambda\rho_*^{\frac1p} : \lambda >0  \}.
\end{align*}
\end{theorem}
This result is quite remarkable. In principle there might be many different $p$-Wasserstein eigendistances for a given Markov chain for a fixed $p$, and certainly for different $p$. But as Theorem \ref{thm:uniqueness} shows, when $E$ is finite and $P$ algebraically irreducible, there is effectively a single spanning $1$-Wasserstein eigendistance $\rho_*$, and any other $p$-Wasserstein eigendistance is directly derived from it by taking the $p$-th root and scalar multiplication. This indicates that the Wasserstein eigendistance is a truly characteristic property of a Markov chain. Also, the perceived freedom of different $p$ does not matter in this context, and we might as well choose $p=1$, since $\kappa_*\geq 1-(1-\kappa_*)^\frac1p$. Also, $p=1$ matches the notion of coarse Ricci curvature from \cite{OLLIVIER:09} and in the examples computations become simpler.

\section{Examples}\label{section:examples}
In this section be will discuss various examples for Markov chains and their eigendistances. We will keep the notation that $P^x$ is the transition kernel, which of course varies from example to example.

\subsection{
\texorpdfstring{Simple symmetric random walk on $\bZ$}
{Simple symmetric random walk on Z}
}
\label{example:ssrw}
Consider a simple symmetric random walk on $\bZ$. Since $\bZ$ is not compact we cannot apply Theorem \ref{thm:eigen-distance} directly. Nevertheless, we can consider the question of eigendistances. 

A first observation is that two random walks with the difference between their starting points an odd number will never meet, since the parity is a preserved quantity under the evolution. It follows that $\rho_2(x,y):=\ind_{x-y = 1 \mod 2} \in E_1(0)\cM_{[0,1]}$.

The eigendistance $\rho_2$ is somewhat disappointing, as it does not capture the typical geometry of $\bZ$, which we associate with the Euclidean distance. In fact, $\rho_0(x,y)=\abs{x-y}$ is also an eigendistance: By Jensen's inequality, 
\[ \cW_{\rho_0}(P^x,P^y) = \inf_{\pi\in\cP_{x,y}}\int \abs{X-Y}d\pi \geq \inf_{\pi\in\cP_{x,y}}\abs{ \int X d\pi - \int Y d\pi}=\abs{x-y}.\]
For a corresponding upper bound we consider a specific coupling, where both chains perform the same jump, either to the right or to the left. Clearly this leaves the distance invariant, and hence $W_1(\rho_0)=\rho_0$.

The fact that there is a distinction between odd and even (euclidean) distances as showcased by $\rho_2$ leads to another possible eigendistance:
\begin{align}
\rho_0'(x,y) :=\begin{cases}
\abs{x-y},\quad &x-y = 0\mod 2;	\\
\infty, &x-y = 1\mod 2.
\end{cases}
\end{align}
Note how $\rho'_0$ is essentially the refinement of $\rho_2$, capturing the finer details of the evolution where $\rho_2$ is not distinguishing points. 

This provides us already with three eigendistances, $\rho_2\in E_1(0)\cM_{[0,1]}$, $\rho_0\in E_1(0)\cM_{(0,\infty)}$ and $\rho_0'\in E_1(0)\cM_{(0,\infty]}$.

\subsection{
\texorpdfstring{Lazy random walk on $\bZ$}
{Lazy random walk on Z}}
In this example we assume that the random walk is lazy, jumping to the left or right with probability $q$ and staying put with probability $r=1-2q$. 

How much does the the option of waiting change the eigendistances of the random walk? First, $\rho_0$ is still an eigendistance to curvature 0, using the same argument.
Also $\rho_2$ is still an eigendistance. 
However, since $x-y \mod 2$ is no longer a preserved quantity
the corresponding curvature is no longer 0.
\begin{lemma}\label{lemma:exampleLRW1}
$W_1(\rho_2) = (1-\kappa_2(r))\rho_2$ with $\kappa_2(r)=1-\abs{2q-r}$.
\end{lemma}
\begin{proof}
Simply note that for a coupling to minimize $\rho_2$, it must maximize the probability that $X_1-Y_1 = 0 \mod 2$. For $x-y = 1 \mod 2$, this is achieved by maximizing the probability that one chain jumps and the other does not move, and this probability is exactly given by $\kappa_2(r)$. When either both chains jump or both stay put we have $X_1-Y_1 = 1 \mod 2$, so that $W_1(\rho_2)(x,y)=0\kappa_2(r) + (1-\kappa_2(r))1 = (1-\kappa_2(r))\rho(x,y)$.
\end{proof}
Note how $\kappa_2(\frac12)=0$, which is a case we excluded in the general discussion. In fact this is an example where part b) of Assumption \ref{as:regularity} is not satisfied.

What is slightly more interesting is the family of eigendistances generalizing $\rho_2$. These eigendistances correspond to the $P$-homomorphisms $x\mapsto x \mod L$, which maps the random walk on $\bZ$ to the random walk on the discrete torus $\bZ/L\bZ$.
\begin{lemma}\label{lemma:circle}
Assume $r>q$. Then, for any $L\in \bN, L\geq 4$, 
\begin{align}\label{eq:LRW:rho_L}
\rho_L(x,y) = \sin\left(\frac{x-y \mod L}{L}\pi\right) \in E_1(\kappa_L(r))\cM_{(0,1]},
\end{align}
with $\kappa_L(r) = (1-r)\left(1-\cos\left(\frac{2}{L}\pi\right)\right)$.
\end{lemma}
\begin{proof}
First note that due to the translation invariance we can restrict ourself to the study of $x-y \mod L$ and do not have to treat all pairs of $x,y$ individually. 

To verify \eqref{eq:LRW:rho_L} we look first at the case $x-y \mod L \in\{2,...,L-2\}$. Here, since $z\mapsto \sin(z\pi/L)$ is concave for $z \in \{0,...,L\}$, an optimal coupling $\pi\in\cP_{x,y}$ for $W_1(\rho_L)(x,y)$ is given by the maximal variance coupling, where both chains jump in opposite directions. The law of $X_1-Y_1$ under this coupling is $q\delta_{x-y-2}+q\delta_{x-y+2} + r\delta_{x-y}$.
To show that indeed $\int \rho_L(X_1,Y_1) d \pi = (1-\kappa_L(r))\rho_L(x,y)$ we use the trigonometric addition formulas. Omitting $\mod L$ for ease of notation,
\begin{align}
\int \rho_L(X_1,Y_1) d\pi    
&= q\sin\left(\frac{x-y+2}{L} \pi \right) + p\sin\left(\frac{x-y-2}{L} \pi \right) +  r\sin\left(\frac{x-y }{L} \pi \right)	\\
&= 2q\sin\left(\frac{x-y}{L} \pi \right) \cos\left(\frac{2}{L} \pi \right) + r\sin\left(\frac{x-y}{L} \pi \right)\\
&= \left(2q\cos\left(\frac{2}{L} \pi \right) + r\right)\rho_L(x,y).
\end{align}
With $2q=1-r$ it follows that indeed $\kappa_L(r)=(1-r)\left(1-\cos\left(\frac{2}{L}\pi\right)\right)$.

What remains is the case when $x-y=\pm 1 \mod L$, and assume w.l.o.g. $x-y= 1 \mod L$. In this case it is optimal to maximize the probability that one chain does not move, while the other jumps on top of it. This can be achieved with probability $2\min(r,q)=2q$. The remaining probability is distributed to maximize the variance. The law of $X_1-Y_1 \mod L$ is given by $2q\delta_0+p\delta_{1+2} + (r-q)\delta_1$. Optimality is again a consequence of concavity of the sine function, this time with the additional constraint that the two random walks do not jump over each other, exchanging positions without reducing distance. To show \eqref{eq:LRW:rho_L},
\begin{align}
&2q\sin\left(\frac0L \pi\right) + q\sin\left(\frac{1+2}L \pi\right) + (r-q)\sin\left(\frac1L\pi\right) \\
&= q\sin\left(\frac1L \pi\right) \cos\left(\frac2L\pi\right)+q\cos\left(\frac1L\pi\right) \sin\left(\frac2L\pi\right) + (r-q)\sin\left(\frac1L\pi\right)	\\
&= \left( q\cos\left(\frac2L\pi\right) + 2q\cos\left(\frac1L\pi\right)^2 + r-q \right)\rho_L(x,y).
\end{align}
With $2\cos(\frac1L \pi)^2 = 1 +\cos(\frac2L \pi)$ the claim follows.
\end{proof}
It may seem somewhat arbitrary to guess the specific form \eqref{eq:LRW:rho_L} of $\rho_L$. It is in fact motivated by Brownian motion on a circle, where an optimal Markovian coupling is given by anti-correlated increments, corresponding to the maximal variance principle. If the Brownian motion has generator $\frac12 \Delta$, then the difference process $X_t-Y_t$ under this coupling is itself Markovian with generator $\Delta$. Since the sine function is an eigenfunction of $\Delta$ and $\sin(0)=0$ it is a natural candidate to obtain an eigendistance for the random walk on a discrete torus.

\subsection{Independent spin flips}\label{example:ISF}
Consider $\{0,1\}^n$ as a system of $n$ spins. At each time step, every spin has a chance $q\in(0,\frac12)$ to flip, independent from the other spins.

The straight-forward guess for an eigendistance is the Hamming distance, which counts the number of disagreeing spins between $x,y\in\{0,1\}^n$. This is indeed an eigendistance, but it can be generalized: For any $a\in[0,\infty)^n$, 
\begin{align}
\rho_a(x,y):=\sum_{i=1}^n a(i)\ind_{x(i)\neq y(i)}\in E_1(2q)\cM_{[0,\infty)}.
\end{align}
This is a direct consequence of Lemma \ref{lemma:exampleLRW1} together with Proposition \ref{prop:tensor}.

\subsection{Markov chain with absorbing states}
Assume that $X_t$ is a Markov chain on a discrete $E$ which is absorbed in $A\subset E$ with $\abs{A}\geq 2$. Then, for any non-trivial partition $A_1 \dot\cup A_2 = A$, consider the hitting times $\tau_i = \inf\{t\geq 0 : X_t \in A_i\}$, $i=1,2$, with the infimum over the empty set being infinite. Assuming $\min(\tau_1,\tau_2)<\infty$ a.s., we have by standard Markov chain theory $h(x):= \bP_x(\tau_1<\tau_2)$ as a harmonic function. By Theorem \ref{thm:eigenfunction} there exists a 1-Wasserstein eigendistance $\rho$ to curvature 0, and for $x\in A_1$, $y\in E\setminus A_1$ we have $\rho(x,y) = \bP_y(\tau_1\leq \tau_2)$.

\subsection{Other examples}
Other examples include the random walk on the discrete hypercube, the discrete Ornstein-Uhlenbeck process and $k$ independent particles on the complete graph, which are example 8, 10 and 12 in \cite{OLLIVIER:09}.

\section{Concentration Estimates}\label{section:concentration}
In this section we explore the way the Markov chain acts on $\rho$-Lipschitz functions for a 1-Wasserstein eigendistance $\rho$, and on $\rho$ itself. In analogy to eigenfunctions and curvature, we will indeed observe various forms of exponential contraction when $\kappa>0$. But also the case $\kappa=0$ shows to be of interest. In a first step we observe that the transition operator $P$ acts as a contraction on the space of $\rho$-Lipschitz functions. For $f:E\to\bR$, define 
\[\norm{f}_{\rho-Lip}:=\inf\{r>0 : \abs{f(x)-f(y)}\leq r \rho(x,y)\ \forall x,y\in E\},  \]
which is a semi-norm on the space of $\rho$-Lipschitz continuous functions.
\begin{lemma}\label{lemma:Lipschitz-contraction}
Let $\rho \in E_1(\kappa) \cM_{[0,\infty)}$. Then $\norm{Pf}_{\rho-Lip}\leq (1-\kappa)\norm{f}_{\rho-Lip}$. In words, the transition operator acts as a contraction on the space of $\rho$-Lipschitz functions. In particular,
\[ \abs{P^tf(x)-P^tf(y)}\leq (1-\kappa)^t\rho(x,y),\quad x,y\in E, t\in\bN. \]
\end{lemma}
\begin{proof}
\begin{align}
\abs{Pf(x)-Pf(y)} &= \abs{\inf_{\pi\in \cP_{x,y}}\int f(X)-f(Y) \;\pi(d(X,Y))} \\
&\leq \norm{f}_{\rho-Lip} \cW_{1,\rho}(P^x,P^y) 
= \norm{f}_{\rho-Lip}(1-\kappa)\rho(x,y). \qedhere
\end{align}
\end{proof}
Motivated by the contraction property of $P$ we turn to a more detailed look the fluctuations of $f(X_t)$ for some $\rho$-Lipschitz continuous function $f$ and arbitrary times $t\in\bN$. The main ingredient to establish good fluctuations results is the following theorem about exponential moments.
\begin{theorem}\label{thm:exponential-moment}
Let $\rho \in E_1(\kappa) \cM_{[0,\infty)}$ and $f:E\to\bR$ $\rho$-Lipschitz continuous. Then, for any $T\in\bN$ and $x\in E$,
\begin{align}
\log \bE_x \exp\left( f(X_T)-P_Tf(x) \right) \leq 
\begin{cases} 
\sum_{n=2}^\infty \frac{\norm{f}_{\rho}^n \sigma^{(n)}_\rho }{(1-(1-\kappa)^{n})n!}, \quad &\kappa>0,\\
T\sum_{n=2}^\infty \frac{\norm{f}_{\rho}^n \sigma^{(n)}_\rho }{n!}, &\kappa=0,
\end{cases}
\end{align}
where 
\[ \sigma^{(n)}_\rho := \sup_{x\in E}\int \left(\int \rho(z,z') P^x(dz')  \right)^n P^x(dz). \]
\end{theorem}
The terms $\sigma_\rho^{(n)}$ measure the intrinsic fluctuations of a single step transition. They can be bounded by $\diam_\rho(E)^n=\sup_{x,y\in E}\rho(x,y)^n$, but typically this is highly inefficient. A more precise bound uses
\begin{align}\label{eq:J}
 J_\rho := \sup_{x\in E} \ \operatorname*{\mathnormal{P}^\mathnormal{x}-ess\,sup}_{y \in E} \rho(x,y), 
\end{align}
the maximal jump the Markov chain can perform as measured by $\rho$. Then
\begin{align}\label{eq:sigma-J}
\sigma^{(n)}_\rho\leq 2^n J_\rho^n
\end{align}
which follows directly from $\rho(z,z')\leq \rho(z,x)+\rho(z',x) \leq 2J_\rho$.

Using \eqref{eq:sigma-J} and the exponential moment bounds in Theorem \ref{thm:exponential-moment} gives control on the fluctuation of $f(X_t)$ via an application of the exponential Markov inequality.
\begin{corollary}\label{corollary:concentration}
Let $\rho \in E_1(\kappa) \cM_{[0,\infty)}$ and assume that $J_\rho<\infty$. Then, for any $\rho$-Lipschitz function $f:E\to\bR$, $x\in E$, $T\in\bN$ and $r>0$,
\begin{align} 
&\bP_x\left( \frac{f(X_T) - \bE_x f(X_T)}{\norm{f}_{\rho-Lip}J_\rho} >  r  \right) \leq \exp\left({-\frac{ r^2}{8(\kappa(2-\kappa))^{-1} + \frac43 r}}\right), \quad &\kappa>0	;	\\ 
&\bP_x\left( \frac{f(X_T) - \bE_x f(X_T)}{\norm{f}_{\rho-Lip}J_\rho T^{\frac12}} >  r  \right) \leq \exp\left(-\frac{ r^2}{8 + \frac{4r}{3T^{\frac12}}}\right), \quad &\kappa=0.
\end{align}
\end{corollary}
Note how for both $\kappa=0$ and $\kappa>0$ there is concentration with a similar upper bound, the main difference being that $\kappa^{-\frac12}$ determines the scale of fluctuations when $\kappa>0$ for any $T$, while for $\kappa=0$ there is a diffusive scaling.

We now change the focus an look at the distance $\rho$ itself. By Theorem \ref{thm:markov-coupling} there is a Markovian coupling $\hP$ realizing the optimal coupling for an eigendistance $\rho$. For a pair $x,y \in E$ we look the fluctuations of $\rho(X_t,Y_t)$, where $X_t$ and $Y_t$ are the realizations of the Markov chain started in $x$ and $y$ under the coupling $\hP$. These mirror the corresponding results for $\rho$-Lipschitz functions, with a slight decrease from the exponential decay rate. 
\begin{theorem}\label{thm:distance-exponential-moment}
Let $\rho \in E_1(\kappa) \cM_{[0,\infty)}$. Then, for any $T\in\bN$, $x,y\in E$ and $\lambda\in \bR$,
\begin{align}
\log\hE_{x,y}\exp\left( \lambda\rho(X_T,Y_T) - \hE_{x,y}\lambda\rho(X_T,Y_T) \right)
&\leq \begin{cases}
2\sum_{n=2}^\infty \frac{|\lambda|^n 2^n\sigma_\rho^{(n)}}{(1-(1-\kappa)^n)n!}, \quad&\kappa>0;\\
2T\sum_{n=2}^\infty \frac{|\lambda|^n 2^n\sigma_\rho^{(n)}}{n!}, \quad&\kappa=0.
\end{cases}
\end{align}
The expectation $\hE$ corresponds to the Markov chain generated by $\hPP$ from Theorem \ref{thm:markov-coupling} corresponding to $\rho$, and $\sigma^{(n)}$ is as in Theorem \ref{thm:exponential-moment}.
\end{theorem}

\begin{corollary}
\label{corollary:distance-concentration}
Let $\rho \in E_1(\kappa) \cM_{(0,\infty)}$ and assume that $J_\rho<\infty$. Then, for any $x,y\in E$, $T\in\bN$ and $r>0$,
\begin{align} 
&\hP_{x,y}\left( \abs{\rho(X_T,Y_T)-(1-\kappa)^T\rho(x,y)}\geq J_\rho r  \right)\leq 2\exp\left(\frac{- r^2}{64(\kappa(2-\kappa))^{-1} + \frac83 r}\right)
\end{align}
when $\kappa>0$. For $\kappa=0$ the corresponding bound is
\begin{align}
&\hP_{x,y}\left( \abs{\rho(X_T,Y_T)-(1-\kappa)^T\rho(x,y)}\geq J_\rho r \right)\leq 2\exp\left(\frac{- r^2}{64T + \frac83 r}\right).
\end{align}
\end{corollary}

\section{Discussion and open problems}\label{section:discussion}
\subsection{Natural distance}
It is in general not clear what a natural notion of distance is. This can depend on context, for example for $E$ a discrete graph the graph distance is a natural distance in many contexts. However, one can also consider the graph as an electrical network and consider the resistance metric, which is a different but also natural choice, and it is related to the behaviour of a random walk on the graph via cover times. On the discrete circle $\bZ/L\bZ$, the resistance metric is given by $d_R(x,y) = \frac{(x-y)(L-(x-y)) \mod L}{L}$, which is (up to scaling) the second order approximation on the Wasserstein distance found in Lemma \ref{lemma:circle}. On the other hand, by Proposition \ref{prop:tensor} the Euclidian distance on $\bZ$ extends to the $\ell_1$ distance as a 1-Wasserstein eigendistance for the random walk on $\bZ^n$, which is very different from the resistance metric.

The notion of a Wasserstein eigendistance $\rho$ provides another notion of natural distance, related to the ease of distinguishing $\bP_x(X_n\in \cdot)$ for different $x$ and $n$. It is well-adapted to obtain concentration estimates for $\rho$-Lipschitz functions, and other consequences from positive curvature discussed in \cite{OLLIVIER:09} follow as well. It is an intrinsic distance to the Markov chain, in the sense that under the conditions of Theorem \ref{thm:uniqueness} it is unique and does not require an a priori metric structure. However, due to its indirect construction via a fixed point argument it is not clear what kind of properties of the Markov chain are encoded in $\rho$.

We should also mention the idea of avoiding the possibly discrete nature of $E$ by working on $\cP(E)$. The space of probability measures is a continuous space, and in \cite{ERBAR:MAAS:12},\cite{FATHI:MAAS:16} a natural metric is constructed on $\cP$ via the dynamics of a reversible Markov chain. Also here motivations comes from notions of geometry and curvature in continuous settings, and positive curvature is related to the contraction rate of the semi-group with respect to relative entropy. It is not directly clear how this is related to Wasserstein eigendistances, but both types of distances are intrinsic to the Markov chain, so there might be a connection. 

\subsection{Non-compact spaces and negative curvature}
In this paper we restrict ourself to compact spaces and non-negative curvature, since Theorem \ref{thm:eigen-distance} relies on compactness and contraction properties for a fixed point argument. As the example of the Euclidean distance for the random walk on $\bZ$ shows there exist also Wasserstein eigendistances beyond this setting. However, general existence in non-compact spaces is a harder problem. In \cite{OLLIVIER:11} it is observed that coarse Ricci curvature is compatible with other notions of negative curvature, for example $\delta$-hyperbolic spaces, which gives hope that Wasserstein eigendistances also exist in such settings.

\subsection{Infinite dimensions}
In a similar vein to the previous subsection, particle systems on for example $\{0,1\}^{\bZ^d}$ are an interesting area to apply the theory of Wasserstein eigendistances. But even though the state space is compact Assumption \ref{as:regularity} is not satisfied, since $x\mapsto P^x$ is not continuous with respect to the total variation distance. Also, here it is not ideal to assume bounded metrics, as is done in Theorem \ref{thm:eigen-distance}. For example in the case of infinitely many independent spin flips, as the limit as $n\to\infty$ of Example \ref{example:ISF}, $\rho_0(x,y) =0$ if $x$ and $y$ disagree at only finitely many sites, and $\rho(x,y) =1$ otherwise, is a Wasserstein eigendistance with curvature 0. But it is clearly a very degenerate pseudo-metric and very uninforming about the dynamics. Instead the Hamming distance $\rho_H(x,y)=\sum_{i}\ind_{x(i)\neq y(i)}$ is a much more informative, has positive curvature, and it can be obtained as the limit of eigendistances of the finite-dimensional systems. One interesting observation is that $\rho_0(x,y) = \ind_{\rho_H(x,y)<\infty}$, so one could argue that $\rho_0$ and $\rho_H$ are in fact the same metric observed at different scales.

\subsection{Continuous time}
Here we made use of the fact that time is discrete, which gives us a minimal time unit for which we optimize in  \eqref{eq:W}. For finite state spaces there is not much difference between a continuous time Markov chain and a lazy discrete time chain. However, for example diffusions have no natural discrete time skeleton. Even defining a Wasserstein eigendistance in continuous time is more delicate. On possibility is to consider a family of fixed point problems via
\begin{align}
\cW_{1,\rho_t}(P_t^x,P_t^y) = e^{-\kappa t}\rho_t(x,y) \quad \forall x,y\in E, t\geq 0.
\end{align}
Ideally, there exists $\rho\in\cM$ so that $\rho_t=\rho$ for all $t\in[0,t_0]$ and some $t_0 \in(0,\infty]$, but this is not guaranteed, and one only has a bound
\begin{align}
\cW_{1,\rho_t}(P_{nt}^x,P_{nt}^y) \leq e^{-\kappa n t}\rho_t(x,y) \quad \forall x,y\in E, t\geq 0, n\in\bN.
\end{align}

\subsection{Finding or approximating Wasserstein eigendistances}
In the examples we used well-understood properties and guesses to find Wasserstein eigendistances as solution to the fixed-point problem $(1-\kappa)^{-1}W(\rho)=\rho$. In general this is not feasible, so one needs better ways to find or at least approximate Wasserstein eigendistances. Theorems \ref{thm:maximal} and \ref{thm:eigenfunction} provide limited ways in this direction, but are very specific to curvature 0 or good eigenfunctions of $P$. Example \ref{example:ssrw} shows that a random walk on the discrete torus has an eigendistance which converges as $L\to\infty$, and one can show that this is indeed a Wasserstein eigendistance for Brownian motion on the continuous torus. This is a specific example where a Markov chain converges to a continuous time Markov process, and the corresponding Wasserstein eigendistances converge to a Wasserstein eigendistance of the limiting process. Under which conditions is such a statement true for other processes?

\section{Proofs}

\subsection{Proof of Theorem \ref{thm:eigen-distance}}
Clearly any $\rho \in E_p(\kappa) \cM_{[0,1]}$ is a fixed point of $(1-\kappa)^{-1}W$. To show the existence of such a fixed point we will be using Schauder's fixed point theorem. To this end we need to find an appropriate compact and convex set. This will be done in several steps.

For $\mu,\nu\in \cP$ we write $\mu\wedge \nu$ for the largest common component of $\mu$ and $\nu$, characterized by
\[ \mu\wedge \nu(A) = \int_A \frac{d\mu}{d(\mu+\nu)}(x)\wedge \frac{d\nu}{d(\mu+\nu)}(x) \;(\mu+\nu)(dx) \]
for any event $A$. We write
\begin{align}
\alpha(x,y):= 1-\abs{P^x\wedge P^y},
\end{align} 
which is the image of the trivial metric $\ind_{x\neq y}$ under $W_1$. In particular $\alpha$ is a metric on $E$.
For $f:E\times E\to \bR$, define the  (semi-)norm $\normb{\cdot}_p$ with the metric structure from $(E,\alpha)^2$ via
\begin{align}
\normb{f}_p^p := \sup_{x_1,x_2,y_1,y_2\in E} \frac{\abs{f(x_1,y_1)^p-f(x_2,y_2)^p}}{\alpha(x_1,x_2)+\alpha(y_1,y_2)}.
\end{align}
Denote the unit ball with respect to $\normb{\cdot}_p$ by
\begin{align}
\cC_p:= \left\{ f:E\times E\to \bR :  \normb{f}_p\leq 1 \right\}.
\end{align}
The corresponding set of norm-bounded metrics is $\cM^p_{[0,1]}:= \cM_{[0,1]}\cap \cC_p$. 

\begin{lemma}\label{lemma:compact}
Assume Assumption \ref{as:regularity}. Then the set $\cM^p_{[0,1]}$
is compact and convex, and its elements are continuous.
\end{lemma}
\begin{proof}
It is straight forward that $\cM^p_{[0,1]}$ is closed and convex. Continuity follows from the fact that the metric $\alpha$ is by Assumption \ref{as:regularity} continuous.

Therefore, relative compactness is sufficient to prove the claim, which will follow from an application of the Arzel\`a-Ascoli theorem. 
Clearly $\cM^p_{[0,1]}$ is pointwise bounded, so we only have to check equicontinuity. 
Fix $x_1,y_1\in E$ and $\epsilon>0$, and choose neighborhoods $U_{x_1},U_{y_1}$ around these points satisfying $\alpha(x_1,U_{x_1}), \alpha(y_1,U_{y_1})\subset[0,\epsilon)$.
Then, for any $\rho\in \cM^p_{[0,1]}$ and $(x_2,y_2)\in U_{x_1}\times U_{y_1}$,
\begin{align} \abs{\rho(x_1,y_1)^p-\rho(x_2,y_2)^p}< 2\epsilon
\end{align}
since $\normb{\rho}_p\leq 1$, and with $\rho\geq0$ this shows equicontinuity.
\end{proof}
Next we will prove a technical lemma, which tells us that a coupling of two probability measures approximately contains a coupling for any sub-probability measures of the marginals.
\begin{lemma}\label{lemma:coupling-decomposition}
Let $\mu,\nu\in\cP$ which each decompose into two sub-probability measures $\mu_1+\mu_2$ and $\nu_1+\nu_2$. Then, for any coupling $\pi\in \cP_{\mu,\nu}$, there exists a sub-probability measure $\pi_1\leq \pi$ with $\pi_1(\cdot,E)\leq \mu_1$, $\pi_1(E,\cdot)\leq \nu_1$ and 
\[\abs{\pi_1}\geq 1-\abs{\mu_2}-\abs{\nu_2}.\]
\end{lemma}
\begin{proof}
Since every probability measure on a Polish space is standard there exists a measurable map $\phi_\mu:[0,1]\to E$ so that 
$\mu=dz \circ \phi_\mu^{-1}$, that is $\mu$ is the image measure of the Lebesgue measure $dz$ on the unit interval under $\phi_\mu$. We can additionally assume $dz|_{[0,\abs{\mu_1}]}\circ \phi_\mu^{-1}=\mu_1$ and $dz|_{[\abs{\mu_1},1]}\circ \phi_\mu^{-1}=\mu_2$, with $dz|_I$ being the restriction of the Lebesgue measure to the interval $I\subset[0,1]$. Similarly there is a $\phi_\nu:[0,1]\to E$ with the corresponding properties for $\nu$. 

The coupling $\pi$ then corresponds to a coupling $\widehat{\pi}$ of two Lebesgue measures on the unit square, so that $\pi= \widehat{\pi}\circ \psi^{-1}$ with $\psi:[0,1]^2\to E^2$, $(z_1,z_2)\mapsto (\phi_\mu(z_1),\phi_\nu(z_2))$. By construction $\mu_1 = \widehat{\pi}|_{[0,|\mu_1|]\times[0,1]}\circ \psi^{-1}(\cdot \times E)$, and similarly for the $\mu_2,\nu_1$ and $\nu_2$. 

Define $\pi_1:= \widehat{\pi}|_{[0,|\mu_1|]\times[0,|\nu_1|]}\circ \psi^{-1}$, which is a sub-probability measure of $\pi$, and clearly
\begin{align}
\pi_1(\cdot \times E) 
= \widehat{\pi}|_{[0,|\mu_1|]\times[0,|\nu_1|]}\circ\psi^{-1} (\cdot\times E)
\leq \widehat{\pi}|_{[0,|\mu_1|]\times[0,1]}\circ\psi^{-1}(\cdot\times E) = \mu_1,
\end{align}
and similarly $\pi_1(E\times \cdot )\leq \nu_1$. Furthermore, $\widehat{\pi}|_{[|\mu_1|,1]\times[0,1]}\circ\psi^{-1}(E\times E) = 1-\abs{\mu_1}=\abs{\mu_2}$ and $\widehat{\pi}|_{[0,1]\times[|\nu_1|,1]}\circ\psi^{-1}(E\times E) = 1-\abs{\nu_1}=\abs{\nu_2}$, implying 
\begin{align}
\abs{\pi_1} = 1 &- \widehat{\pi}|_{[|\mu_1|,1]\times[0,1]}\circ\psi^{-1}(E\times E) - \widehat{\pi}|_{[0,1]\times[|\nu_1|,1]}\circ\psi^{-1}(E\times E) \\
&+ \widehat{\pi}|_{[|\mu_1|,1]\times[|\nu_1|,1]}\circ\psi^{-1}(E\times E) \geq 1-\abs{\mu_2}-\abs{\nu_2}.	\qedhere
\end{align}
\end{proof}
\begin{lemma}\label{lemma:F-Lipschitz}
The map $W_p$ satisfies $\sup_{\rho\in\cM_{[0,1]}} \normb{W_p(\rho)}_p\leq 1$.
\end{lemma}
\begin{proof}
Fix $\rho\in \cM^p_{[0,1]}$. We have to show that $W_p(\rho)$ satisfies 
\begin{align}
\abs{W_p(\rho)^p(x_1,y_1)-W_p(\rho)^p(x_2,y_2)}\leq \alpha(x_1,x_2)+\alpha(y_1,y_2).
\end{align}
for any $x_1,x_2,y_1,y_2\in E$. To do so, assume w.l.o.g. that $W_p(\rho)(x_1,y_1)\geq W_p(\rho)(x_2,y_2)$ and let $\pi^{x_2,y_2}\in \cP_{x_2,y_2}$ be an optimal coupling with respect to $\rho$. 
By writing $P^{x_2}=P^{x_1}\wedge P^{x_2}+ (P^{x_2}-P^{x_1}\wedge P^{x_2})$ and similar for $P^{y_2}$ we find by Lemma \ref{lemma:coupling-decomposition} a sub-probability measure $\pi_1\leq \pi^{x_2,y_2}$. 
Using the properties of $\pi_1$ we have
\begin{align}\label{eq:pi-approx}
 \pi_1+ (1-\abs{\pi_1})^{-1}(P^{x_1}-\pi_1(\cdot,E))\otimes (P^{y_1}-\pi_1(E,\cdot)) \in \cP_{x_1.y_1} 
\end{align}
and
\begin{align}
&\abs{W_p(\rho)^p(x_1,y_1)-W_p(\rho)^p(x_2,y_2)} \\
&= \inf_{\pi\in\cP_{x_1,y_1} }\int \rho^p(X,Y) d\pi(X,Y)-\int \rho^p(X,Y) d\pi^{x_2,y_2}(X,Y)	\\
&\leq \int \rho^p(X,Y) d\left[(1-\abs{\pi_1})^{-1}(P^{x_1}-\pi_1(\cdot,E))\otimes (P^{y_1}-\pi_1(E,\cdot)\right](X,Y)	\\
&\leq 1-\abs{\pi_1} \leq 1-\abs{P^{x_1}\wedge P^{x_2}} + 1- \abs{P^{y_1}\wedge P^{y_2}}=\alpha(x_1,x_2)+\alpha(y_1,y_2).\qedhere
\end{align}
\end{proof}

\begin{lemma}\label{lemma:F-continuous}
The map $W_p$ satisfies $\norm{W_p(\rho_1)^p-W_p(\rho_2)^p}_\infty \leq \norm{\rho_1^p-\rho_2^p}_\infty$. In particular $W$ is continuous.
\end{lemma}
\begin{proof}
Fix $\rho_1,\rho_2 \in \cM^\alpha_{[0,1]}$ and $x,y\in E$. W.l.o.g. assume $W_p(\rho_1)(x,y)\geq W_p(\rho_2)(x,y)$ and let $\pi^{x,y}\in \cP_{x,y}$ be an optimal coupling for $W_p(\rho_2)(x,y)=\cW_{p,\rho_2}(P^x,P^y)$. Then
\begin{align}
&\abs{W_p(\rho_1)^p(x,y)-W_p(\rho_2)^p(x,y)} \\
&= \inf_{\pi\in\cP_{x,y}}\int \rho_1^p(X,Y)d\pi(X,Y)-\int \rho_2^p(X,Y)d\pi^{x,y}(X,Y)	\\
&\leq \int \rho_1^p(X,Y)-\rho_2^p(X,Y)d\pi^{x,y}(X,Y) \\
&\leq \norm{\rho_1^p-\rho_2^p}_{\infty}.	\qedhere
\end{align}
\end{proof}

\begin{proof}[Proof of Theorem \ref{thm:eigen-distance}]
Fix $\orho\in\cM_{[0,1]}$ non-degenerate with $W_p(\orho)\leq \orho$ and \eqref{eq:orho-inf}.  

For $\rho \in \cM_{[0,\orho]}$, let $\lambda(\rho):=\sup_{x,y:\orho(x,y)>0}\frac{\rho(x,y)}{\orho(x,y)}\in [0,1]$ measure the scale of $\rho$ with respect to $\orho$. Note that the function $\lambda$ is continuous, since
\begin{align}
&\abs{\lambda(\rho_1)-\lambda(\rho_2)}
\overset{\text{\tiny w.l.o.g.}}{=}
\sup_{x,y} \frac{\rho_1(x,y)}{\orho(x,y)} - \sup_{x,y} \frac{\rho_2(x,y)}{\orho(x,y)}\\
&\leq 
\sup_{x,y} \frac{\rho_1(x,y) - \rho_2(x,y)}{\orho(x,y)} \leq \frac{\norm{\rho_1-\rho_2}_\infty}{\inf_{x,y:\orho(x,y)>0}\orho(x,y)}.
\end{align}
Also clearly $\lambda(\rho)=0$ iff $\rho=0$.

Define $\cM_{[0,\orho]}^p:=\cM_{[0,\orho]}\cap \cM_{[0,1]}^p$, which is convex and compact by Lemma \ref{lemma:compact}.

Consider now the sets 
\[ B_1 := \{W_p(\rho) : \rho \in \cM_{[0,\orho]},\lambda(\rho)=1 \} \]
and $B_2:=\overline{\operatorname{conv}(B_1)}$ its closed convex hull. By Lemma \ref{lemma:F-Lipschitz} and $W_p(\orho)\leq \orho$ we have $B_1\subset \cM_{[0,\orho]}^p:=\cM_{[0,\orho]}\cap \cM_{[0,1]}^p$, and by convexity of $\cM_{[0,\orho]}^p$ also $B_2\subset \cM_{[0,\orho]}^p$, which makes $B_2$ compact as well. Since $0$ is an extremal point of $\cM^p_{[0,\orho]}$ it follows that $0\in B_2$ if and only if $0\in \overline{B_1}$. For this to be the case there has to exist a sequence $\rho_n$ with $\lambda(\rho_n)=1$ so that $W_p(\rho_n)\to 0$. But then, by Assumption \ref{as:regularity}, $\rho_n\to 0$, which leads to a contradiction since $\lambda$ is continuous and $\lim\lambda(\rho_n)=1\neq0=\lambda(0)$. Therefore $0\not\in B_2$.

On the convex compact set $B_2$ define the map $F:B_2\to B_2$, 
$
\rho \mapsto W_p\left(\frac{\rho}{\lambda(\rho)}\right).
$
It is well-defined since $0\not\in B_2$, and continuous as concatenation of continuous functions. By Schauder's fixed point theorem there exists a fixed point $\rho_*\in B_2$ of $F$ .
With $\kappa_*:=1-\lambda(\rho_*)$ we see that $W_p(\rho_*)=(1-\kappa_*)^{-1}\rho_*$, proving the theorem.
\end{proof}

\subsection{Proof of Theorem \ref{thm:maximal} and Theorem \ref{thm:markov-coupling}}
The proof of Theorem \ref{thm:maximal} is one half of the Knaster-Tarski fixed point theorem for complete lattices. Note that $\cM_{[0,1]}$ is not a complete lattice, since the minimum of two pseudo-metrics need not be a pseudo-metric.
\begin{proof}[Proof of Theorem \ref{thm:maximal}]
Set $B:=\{ \rho\in\cM_{[0,1]} : \rho\leq W_p(\rho) \}\in \cM_{[0,1]}$ and let $\rho_*:=\bigvee B$, where $\bigvee B$ denotes the pointwise supremum over all elements in the set $B$. Note that the set is not empty since $0\leq W_p(0)$. The claim is that $\rho_*$ is a fixed point, which would make it the largest fixed point since all fixed points are contained in $B$. To see this, by construction $\rho\leq \rho_*$ for any $\rho\in B$, and by monotonicity $W_p(\rho)\leq W_p(\rho_*)$. Since $\rho\in B$ this implies $\rho\leq W_p(\rho_*)$ for all $\rho\in B$. But $\rho_*$ is the smallest upper bound on $B$, so that $\rho_*\leq W_p(\rho_*)$ must hold. By monotonicity then $W_p(\rho_*)\leq W_p(W_p(\rho_*))$ and therefore $W_p(\rho_*)\in B$. But $\rho_*$ is an upper bound on $B$, hence $W_p(\rho_*)\leq \rho_*$ and we have $\rho_*=W_p(\rho_*)$.

To see that $\lim_{n\to\infty}W_p^n(\ind_{x\neq y})=\rho_*$, we  use the monotonicity of $W$. The sequence $W_p^n(\ind_{x\neq y})$ is decreasing in $n$ and hence converging to a fixed point $\rho_1\in\cM_{[0,1]}$ of $W$. Since $\rho_*$ is the maximal fixed point, $\rho_1\leq \rho_*$. But $\ind_{x\neq y}\geq \rho_*$ and hence $W_p^n(\ind_{x\neq y})\geq \rho_*$, which shows the claim.
\end{proof}

\begin{proof}[Proof of Theorem \ref{thm:eigenfunction}]
Write $\rho_1(x,y):= \abs{h(x)-h(y)}^\frac1p$ and $\rho_2(x,y) := \ind_{x\neq y}(h(x) + h(y))^\frac1p$, which are both metrics. We claim that  $\lambda^{-1}W_p(\rho_1)\geq \rho_1$ and $\lambda^{-1}W_p(\rho_2)\leq \rho_2$, from which the existence of a fixed point in $\cM_{[\rho_1,\rho_2]}$ follows via the proof of Theorem \ref{thm:eigen-distance}, since $F(\cM_{[\rho_1,\rho_2]})\subset \cM_{[\rho_1,\rho_2]})$. To show the claim,
\begin{align}
&W_p(\rho_1)^p(x,y) = \inf_{\pi\in\cP_{x,y}} \int \abs{h(X)-h(Y)} \pi(d(x,y)) \\
&\geq \abs{ \inf_{\pi\in\cP_{x,y}} \int h(X)-h(Y) \pi(d(x,y))} = \abs{Ph(x)-Ph(y)} = 
\lambda\rho_1^p
\end{align}
and
\begin{align}
&W_p(\rho_2)^p(x,y) = \ind_{x\neq y} \inf_{\pi\in\cP_{x,y}} \int \ind_{X_\neq Y}(h(x)+h(y)) \pi(d(x,y)) \\
&\leq \ind_{x\neq y} \inf_{\pi\in\cP_{x,y}} \int h(X)+h(Y) \pi(d(x,y)) = \ind_{x\neq y}(Ph(x)+Ph(y)) = \lambda\rho_2^p. \qedhere
\end{align}
\end{proof}

Before proving Theorem \ref{thm:markov-coupling} we need the following technical lemma. 
\begin{lemma}\label{lemma:weakly-measurable}
Let $\phi:E \to F$ be a $P$-homomorphism and $D_\phi:=\{(x,y)\in E\times E : \phi(x)=\phi(y)  \}$. Write $\cP_{x,y}(D_\phi):=\{\pi\in \cP_{x,y} : \pi(D_\phi)=1 \}$ for the couplings concentrated on $D_\phi$.

For any $A\subset \cP(E^2)$ closed the set 
\[ A^\phi:=\{ (x,y)\in E^2 : \cP_{x,y}(D_\phi)\subset A \} \]
is closed.
\end{lemma}
\begin{proof}
Assume $A$ non-empty, otherwise the statement is trivial.
Consider a converging sequence $(x_n,y_n)_{n\in\bN}\subset A^\phi$, with limit $(x,y)$, and $\pi\in \cP_{x,y}(D_\phi)$. By writing $P^{x}=P^{x_n}\wedge P^{x}+ (P^{x}-P^{x_n}\wedge P^{x})$ and similar for $P^{y}$ we find by Lemma \ref{lemma:coupling-decomposition} a sub-probability measure $\pi_1\leq \pi$ with the listed properties. Writing $(\phi,\phi):(x,y)\mapsto (\phi(x),\phi(y)$ and $\diag(F)$ for the diagonal in $F^2$ we see that $\supp(\pi\circ (\phi,\phi)^{-1})\subset \diag(F)$, and by virtue of $\pi_1\leq \pi$ the same holds true for $\pi_1$. Consequently we have $\pi_1(\cdot \times E)\circ \phi^{-1} = \pi_1\circ \pi_1(E\times \cdot)\circ \phi^{-1}$. It follows that
\[ pi_n:=\pi_1 + (1-\abs{\pi_1})^{-1}[P^{x_n}-\pi_1(\cdot\times E))\otimes(P^{y_n}-\pi_1(E\times \cdot)] \in \cP_{x_n,y_n}(D_\phi)\subset A, \]
and
\[ 
\abs{\pi-\pi_n} =  1-\abs{\pi_1} \leq 1-\abs{P^{x_n}\wedge P^x}+1-\abs{P^{y_n}\wedge P^y} = \alpha(x,x_n) + \alpha(y,y_n).
\]
Since the right hand side converges to 0 and $A$ is closed it follows that $\pi\in A$. As $\pi \in \cP_{x,y}(D_\phi)$ was arbitrary it follows $\cP_{x,y}(D_\phi)\subset A$ and hence $(x,y)\in A^\phi$.
\end{proof}

\begin{proof}[Proof of Theorem \ref{thm:markov-coupling}]
By Theorem \ref{thm:eigen-distance} $\rho$ is continuous. Consider the multi-valued map or correspondence $\Psi:E^2 \twoheadrightarrow \cP$, $(x,y)\mapsto \cP_{x,y}$. Here we refer to the appendix for the minimal needed facts about multi-maps or correspondences used. Clearly $\Psi$ has non-empty compact values.
Let $\phi:E\to\{0\}$ be the trivial constant $P$-homomorphism. Note that $\Psi^u(A) = A^\phi$ from Lemma \ref{lemma:weakly-measurable} and hence $\Psi$ is weakly measurable.

Consider a second correspondence $\Pi:E^2 \twoheadrightarrow \cP$ given by 
\[\Pi(x,y) = \{ \pi\in \Psi(x,y) : \int \rho^p d\pi = W_p(\rho)^p \}. \]
By the 
Kuratowski-Ryll-Nardzewski Selection Theorem \ref{thm:selection} there exists a measurable selection of $\Pi$, meaning a measurable function $\hPP : E^2 \to \cP$ with $\hPP(x,y)\in \Pi(x,y)$, which is in the notation $\hPP^{x,y}$ the desired Markov transition kernel of the coupling realising $\rho$, and 
\begin{align}
\hE_{x,y}\rho^p(X_n,Y_n) &= \hE_{x,y} \hPP^{X_{n-1},Y_{n-1}}(\rho^p) = (1-\kappa)^p\hE_{x,y}\rho^p X_{n-1},Y_{n-1}\\
&=...=  (1-\kappa)^{pn}\rho^p(x,y).	\qedhere
\end{align}
\end{proof}

\subsection{Proofs regarding algebraic irreducibility}

\begin{proof}[Proof of Proposition \ref{prop:tensor}]
For $\rho=\rho_X$ we have
\begin{align}
W_p(\rho)^p((x,y),(u,v)) &= \inf_{\pi\in \cP_{(x,y),(u,v)}}\int \rho_X^p(X,U) \,\pi\left(d((X,Y),(U,V))\right)\\
&=\inf_{\pi'\in \cP_{x,u}} \int \rho_X^p(X,U) \,\pi'\left(d(X,U)\right) = W_p(\rho_X)^p(x,u).
\end{align}
The proof for $\rho=\rho_Y$ is analogous.

Assume now $\kappa_X=\kappa_Y=\kappa$.
Let $\pi_X^{x,u}$, $\pi_Y^{y,v}$ be optimal couplings for $\rho_X$ and $\rho_Y$, so that $\int \rho_X^p d\pi_X^{x,u} = (1-\kappa)^p\rho_X^p(x,u)$ and $\int \rho_Y^p d\pi^{y,v}_Y = (1-\kappa)^p\rho_Y^p(y,v)$. Then
\begin{align}
&\inf_{\pi\in \cP_{(x,y),(u,v)}} \int \rho^p\ d\pi 
\leq \int \rho^p\ d\left(\pi_X^{x,u}\otimes \pi_Y^{y,v}\right) 	
= a\int \rho_X^p\ d\pi_X^{x,u} + b\int \rho_Y^p\ d\pi^{y,v}_Y \\
&\quad = a(1-\kappa)^p\rho_X^p(x,u) + b(1-\kappa)^p\rho_Y^p(y,v)
=(1-\kappa)^p\rho^p((x,y),(u,v)).
\end{align}
For a lower bound,
\begin{align}
& \inf_{\pi\in \cP_{(x,y),(u,v)}} \int \rho^p d\pi 
\geq a\inf_{\pi\in\cP_{x,u}} \int \rho_X^p d\pi + b\inf_{\pi\in\cP_{y,v}}\int \rho_Y^p d\pi	\\
&= a(1-\kappa)^p\rho_X^p(x,u) + b(1-\kappa)^p\rho_Y^p(y,v)
= (1-\kappa)^p\rho^p((x,y),(u,v)).	\qedhere
\end{align}
\end{proof}

\begin{proof}[Proof of Proposition \ref{prop:markov-galaxies}]
First note that $E/\sim$ with metric $\rho_\sim([x],[y])=\rho(x,y)$ is a compact Polish space, and $x\mapsto[x]$ is continuous. 
Let $x,y\in E$ with $[x]=[y]$. Since $\rho$ is an eigendistance, $W_p(\rho)(x,y)=(1-\kappa)\rho(x,y)=0$. This implies that for any optimizer $\pi\in\cP_{x,y}$ of the Wasserstein distance $\int \ind_{[X]=[Y]} \;d\pi(X,Y)=1$. In particular it follows that 
\begin{align}\label{eq:class-markov}
\bP_x([X_1]\in d[z]) = \bP_y([X_1]\in d[z]).
\end{align}
With the Markov property of $X$ we get for any $[x_0],...,[x_t]\in E/\sim$ and any representative $x_0\in E$ of $[x_0]$
\begin{align}
&\bP_{x_0}([X_t]\in d[z] \;|\; [X_{t-1}]=[x_{t-1}],...,[X_1]=[x_1] ) 	\\
&= \bE_{x_0}\left(\bP_{X_{t-1}}([X_1]\in d[z]) \;\middle|\; [X_{t-1}]=[x_{t-1}],...,[X_1]=[x_1] \right) 
\end{align}
Since $[X_{t-1}]=[x_{t-1}]$ and \eqref{eq:class-markov}, the above simply equals $\bP_{x_{t-1}}([X_1]\in d[z])$ for any representative $x_{t-1}$, independent of the choice of  the representative $x_0$.
\end{proof}

\begin{proof}[Proof of Theorem \ref{thm:irreducibility1}]
For the proof it is convenient to work with reducibility instead of irreducibility. Therefore we will show the negated equivalence.
\par{$\neg b)\Rightarrow \neg a)$} Let $\rho\in E_p([0,1])M_{[0,1)}$ be a true pseudo-metric. By Proposition \ref{prop:markov-galaxies} the map $\phi(x):=[x]$ is a $P$-homomorphism. Note that by Theorem \ref{thm:eigen-distance} $\rho$ is continuous and therefore $E/\sim$ is a compact Polish metric space with metric $\rho_\sim([x],[y])=\rho(x,y)$. The map $\phi$ is not injective since $\rho$ is a true pseudo-metric. Furthermore, $\phi$ is not constant since $\rho$ is non-degenerate. Therefore $P$ is not algebraically irreducible.


\par{$\neg a) \Rightarrow \neg c)$} Assume $P$ reducible. Then there exists a non-constant $P$-homomorphism $\phi$ with $\phi(x_0)=\phi(y_0)$ for some $x_0,y_0\in E$, $x_0\neq y_0$. Denote by $D_\phi:=\{(x,y)\in E\times E : \phi(x)=\phi(y)\}$ the set which maps onto the diagonal under $\phi$. 
For any $(x,y) \in D_{\phi}$ we have $P^x \circ \phi^{-1}= P^y \circ \phi^{-1}$, hence there exist couplings $\pi\in \cP_{x,y}$ with $\pi(D_\phi)=1$. Define
the correspondence $\Pi:E^2 \twoheadrightarrow \cP(E^2)$ via
\begin{align}
\Pi(x,y)=
\begin{cases}
\{\pi\in \cP_{x,y} : \pi(D_\phi)=1 \} ,& \quad (x,y)\in D_\phi\setminus D;	\\
\{\int \delta_{X,X} P^x(dX)\},&\quad x=y;	\\
\cP_{x,y},&\quad \text{otherwise}.
\end{cases}
\end{align}
To see that $\Pi$ is weakly measurable, let $A\subset \cP(E^2)$ be closed and note that
\begin{align}
\Pi^u(A) =& 
\left\{ (x,y)\in D_\phi : \cP_{x,y}(D_\phi)\subset A \right\}
\cup
\left\{ (x,x) \in D : \int \delta_{X,X} P^x(dX)\in A \right\}	\\
&\cup 
\left\{(x,y)\in E^2\setminus D_\phi : \cP_{x,y}\subset A \right\} =:B_1 \cup B_2 \cup B_3 .
\end{align}
By Lemma \ref{lemma:weakly-measurable} $A^\phi$ is closed, and since $\phi$ is continuous so is $B_1=A^\phi\cap D_\phi$. The set $B_2$ is closed as the pre-image of $A$ of the continuous map $(x,x)\mapsto \int \delta_{X,X} P^x(dX)$. For $B_3$ we use Lemma \ref{lemma:weakly-measurable} applied to the constant $P$-homomorphism $x\mapsto 0$ and $B_3= A^{x\mapsto 0}\cap D_\phi^c$, which is measurable as an intersection of a closed and an open set. Therefore $\Pi$ is weakly measurable, and by Theorem \ref{thm:selection} there exists a measurable selection in the form of a coupling operator $\hPP$. We can assume $\hPP$ to be symmetric by the remark following Definition \ref{def:symmetric-coupling}. 
Then we have $(\hPP^{x,y})^n(D_{\phi})=1$ for any $(x,y)\in D_\phi$ and $n\in\bN$. But $\phi$ is not constant and hence $D_\phi^c$ is an open non-empty subset of $E\times E$ with $(\hPP^{x,y})^n(D_{\phi}^c)=0$ for all $(x,y)\in D_\phi$. Therefore $\hPP$ is a non-irreducible symmetric coupling.

\par{$\neg c) \Rightarrow \neg b)$} 
Let $\hPP$ be a non-irreducible coupling operator. Therefore $(E\times E)_{sym}\setminus D$ can be partitioned into two sets $I$ and $J$ with non-empty interior so that $I$ does not communicate with $J$. That is $\left(\hPP_{sym}^{\{x,y\}}\right)^n(J)=0$ for all $\{x,y\}\in I$, $n\in \bN$. Consider $\overline{\rho}(x,y):=(\ind_{x\neq y} - \ind_{\{x,y\}\in I})\in\cM_{[0,1]}$. 
For any $\{x,y\}\in I$,
\begin{align}
W_p(\overline\rho)^p(x,y) \leq \hPP^{x,y}\overline\rho^p = 0,
\end{align}
therefore $W_p(\overline\rho)\leq \overline\rho$. By Theorem \ref{thm:eigen-distance} there exists a $p$-Wasserstein eigendistance $\rho$ of $P$ in $\cM_{[0,\overline\rho]}$ which is is true pseudo-metric since $I$ is non-empty.
\end{proof}

The main effort in proving Theorem \ref{thm:uniqueness} is the following lemma.
\begin{lemma}\label{lemma:uniqueness}
Assume $E$ finite and $P$ algebraically irreducible. Then, for any two $\rho,\rho'\in E_p([0,1))\cM_{[0,1]}$ there exists $\lambda>0$ so that $\rho'=\lambda\rho$.
\end{lemma}
\begin{proof}
Let $\rho\in E_p(\kappa)\cM_{[0,1]}$ and $\rho'\in E_p(\kappa')\cM_{[0,1]}$ be two eigendistances. By Theorem \ref{thm:irreducibility1} both are proper metrics. Let  $\lambda:=\sup_{x}\sup_{y\neq x} \frac{\rho'(x,y)}{\rho(x,y)}$, which is finite since $E$ is finite. Then $\lambda\rho\geq \rho'$ and $\lambda\rho(x_0,y_0)=\rho'(x_0,y_0)>0$ for some $x_0,y_0\in E$. 
Then, since $\rho,\rho'$ are eigendistances and  $W_p$ in monotone,
\begin{align}
W_p(\lambda\rho)(x_0,y_0)&=(1-\kappa)\lambda\rho(x_0,y_0)= (1-\kappa)\rho'(x_0,y_0) \\
&= \frac{1-\kappa}{1-\kappa'}W_p(\rho')(x_0,y_0)\leq \frac{1-\kappa}{1-\kappa'}W_p(\lambda\rho)(x_0,y_0), 
\end{align}
 which implies $\kappa'\geq \kappa$. By reversing the roles of $\rho$ and $\rho'$ the above argument implies $\kappa'\leq \kappa$ as well, and hence $\kappa'=\kappa$.

Assume that $\lambda\rho\neq \rho'$. Then there exist $x_1,y_1\in E$ so that $\lambda\rho(x_1,y_1)>\rho'(x_1,y_1)$. By Theorem \ref{thm:markov-coupling} there exists a coupling operator $\hPP$ realising $W_p$. By Theorem \ref{thm:irreducibility1} $\hPP$ is irreducible and there exists $k\in\bN$ so that $\hP_{x_0,y_0}(\{X_k,Y_k\}=\{x_1,y_1\})>0$. Then
\begin{align}
(1-\kappa)^{pk}\rho'(x_0,y_0)^p &= (1-\kappa)^{pk}\lambda^p\rho(x_0,y_0)^p
= \hE_{x_0,y_0}\lambda^p\rho(X_k,Y_k)^p 		\\
&> \hE_{x_0,y_0}\rho'(X_k,Y_k)^p = \left((\hPP)^k(\rho')^p\right)(x_0,y_0)\\
&\geq (W_p^k(\rho'))^p(x_0,y_0) = (1-\kappa)^{pk}\rho'(x_0,y_0)^p.
\end{align}
The last line follows from the fact that $W$ is the infimum over all possible couplings. As we have shown a contradiction, this implies that $\lambda\rho=\rho'$, showing uniqueness of the $p$-Wasserstein eigendistance up to multiplication.
\end{proof}

\begin{proof}[Proof of Theorem \ref{thm:uniqueness}]
By Theorem \ref{thm:eigen-distance} there exists $\rho_* \in E_1([0,1])$ with curvature $\kappa_*$. By Theorem \ref{thm:irreducibility1} it is a proper metric. Since $(a+b)^\frac1p\leq a^\frac1p + b^\frac1p$ it follows that $\rho_*^\frac1p$ satisfies the triangle inequality and hence is a proper metric too. It is a $p$-Wasserstein eigendistance with curvature $1-(1-\kappa_*)^\frac1p$ since
\[ W_p^p(\rho_*^\frac1p)(x,y) = \inf_{\pi\in\cP_{x,y}}\int \rho_* d\pi = W_1(\rho_*) = (1-\kappa_*)\rho_*(x,y). \]
By Lemma \ref{lemma:uniqueness}, any other $p$-Wasserstein eigendistance is a scalar multiple of $\rho_*^frac1p$.
\end{proof}

\section{Proofs of concentration}

\begin{proof}[Proof of Theorem \ref{thm:exponential-moment}]
Fix $f$ and $T$. We will first prove that
\begin{align}\label{eq:concentration1}
&\bE\left[ e^{P_{T-t-1}f(X_{t+1})-P_{T-t}f(X_t)} \;\middle|\; \cF_t \right]	\\ \label{eq:concentration2}
&\leq 1+ \sum_{n=2}^\infty \frac{1 }{n!} (1-\kappa)^{(T-t-1)n}\norm{f}_{\rho-Lip}^n \sigma_\rho^{(n)},
\end{align}
with $\cF_t$ the canonical filtration. Writing $e^x=\sum_{n=0}^\infty \frac{x^n}{n!}$ we first observe that the term for $n=1$ vanishes in \eqref{eq:concentration1} via  an application of the Markov property. To match the terms for $n\geq2$ with the right hand side we use the estimate
\begin{align}
&\bE\left[ \left(P_{T-t-1}f(X_{t+1})-P_{T-t}f(X_t)\right)^n \;\middle|\; \cF_t \right]	\\
&\leq \sup_{x\in E} \int \left(P_{T-t-1}f(z)-P_{T-t}f(x)\right)^n P^x(dz).
\end{align}
Since $f$ is $\rho$-Lipschitz we can use Lemma \ref{lemma:Lipschitz-contraction} to obtain
\begin{align}
\int \left(P_{T-t-1}f(z)-P_{T-t}f(x)\right)^n P^x(dz)
\leq (1-\kappa)^{(T-t-1)n}\norm{f}_{\rho-Lip}^n \sigma_\rho^{(n)},
\end{align}
which shows \eqref{eq:concentration2}.

By telescoping, $f(X_T)-P_Tf(x) = \sum_{t=0}^{T-1}P_{T-t-1}f(X_{t+1})-P_{T-t}f(X_t)$, and with \eqref{eq:concentration2}
\begin{align}
\bE_x e^{f(X_T)-P_Tf(x)} 
& \leq \prod_{t=0}^{T-1} \left(1+ \sum_{n=2}^\infty \frac{1 }{n!} (1-\kappa)^{(T-t-1)n}\norm{f}_{\rho-Lip}^n \sigma_\rho^{(n)}\right)	\\
&\leq \exp\left( \sum_{n=2}^\infty \frac{1 }{n!} \sum_{t=0}^{T-1}(1-\kappa)^{(T-t-1)n}\norm{f}_{\rho-Lip}^n \sigma_\rho^{(n)}  \right).
\end{align}
The sum over $t$ is bounded by $(1-(1-\kappa)^n)^{-1}$ in the case $\kappa>0$, and by $T$ when $\kappa=0$.
\end{proof}

\begin{lemma}\label{lemma:concentration}
Assume a random variable $Z$ with $\bE Z =0$ satisfies
\begin{align}\label{eq:concentration-lemma}
\log\bE e^{\lambda Z}\leq \alpha \sum_{n=2}^\infty \frac{|\lambda|^n\beta^n}{n!}
\end{align}
for some $\alpha, \beta>0$ and all $\lambda \in \bR$. Then, for any $r\geq 0$,
\begin{align}\label{eq:concentration-lemma2}
\bP(\pm Z \geq r\beta) \leq \exp\left(\frac{-\frac12 r^2}{\alpha+\frac13 r}\right).
\end{align}
\end{lemma}
\begin{proof}
For any $\lambda>0$, by the exponential Markov inequality and \eqref{eq:concentration-lemma}, 
\begin{align*}
\bP_x(\pm Z > r\beta ) 
&\leq e^{-\lambda r\beta} \bE_x e^{\pm \lambda Z}  
\leq \exp\left(-\lambda r\beta+ \alpha \sum_{n=2}^\infty\frac{\lambda^n\beta^n}{n!} \right)	.
\end{align*}
Through optimizing $\lambda$, the exponent becomes
\begin{align}\label{eq:optimized-exponent}
r-\left(\alpha+r\right)\log\left(\frac{r}{\alpha}+1\right).
\end{align} 

To show $\eqref{eq:optimized-exponent}\leq \frac{-\frac12 r^2}{\alpha+\frac13 r}$ is equivalent to showing
\[ \log\left(\frac{r}{\alpha}+1\right) \geq \frac{\frac{\frac12 r^2}{\alpha+\frac13r}+r}{\alpha+r}. \]
Through comparing the derivatives w.r.t. $r$ we conclude that the left hand side is indeed bigger than the right hand side, which proves the lemma.
\end{proof}

\begin{proof}[Proof of Corollary \ref{corollary:concentration}]
Write $\beta=\norm{f}_{\rho-Lip}2J_\rho$ and $\alpha=T$ or $\alpha=(\kappa(2-\kappa))^{-1}=(1-(1-\kappa)^2)^{-1}$, depending on $\kappa=0$ or $\kappa>0$. By Theorem \ref{thm:exponential-moment} and the remark following it the conditions of Lemma \ref{lemma:concentration} are satisfied, and
\begin{align*}
\bP_x(f(X_T)-\bE_x f(X_T) > r\beta ) \leq \exp\left( \frac{-\frac12 r^2}{\alpha+\frac13 r} \right).
\end{align*}
In the case $\kappa>0$ this is the claim, for $\kappa=0$ use $r=2T^{\frac12}\tilde{r}$.
\end{proof}

\begin{proof}[Proof of Theorem \ref{thm:distance-exponential-moment}]
In a first step we will show that
\begin{align}\label{eq:dist-exp-moment-1}
\hE_{x,y}e^{\lambda(\rho(X_1,Y_1)-\hPP\rho(x,y)}
&\leq 1+ 2\sum_{n=2}^\infty\frac{|2\lambda|^n\sigma^{(n)}}{n!} .
\end{align}
Writing the exponential as a series, the term for $n=1$ vanishes. For $n\geq 2$, 
\begin{align}
&\abs{\hE_{x,y}\left(\rho(X_1,Y_1)-\hPP\rho(x,y) \right)^n} 		\\
&\leq \hE_{x,y}\left(\int \abs{\rho(X_1,Y_1)-\rho(u,v)}\;\hPP^{x,y}(d(u,v)) \right)^n		\\
&\leq \hE_{x,y}\left(\int \rho(X_1,u)+\rho(Y_1,v)\;\hPP^{x,y}(d(u,v)) \right)^n		\\
&\leq 2^n\hE_{x,y}\left(\int \rho(X_1,u)\;\hPP^{x,y}(d(u,v)) \right)^n + 2^n\hE_{x,y}\left(\int \rho(Y_1,v)\;\hPP^{x,y}(d(u,v)) \right)^n		\\
&= 2^n\bE_{x}\left(\int \rho(X_1,u)\;P^{x}(du)) \right)^n + 2^n\hE_{y}\left(\int \rho(Y_1,v)\;P^{y}(dv)) \right)^n 
\leq 2^{n+1}\sigma^{(n)}.	\label{eq:rho-moment}
\end{align}

In a second step, we use the Markov property,  \eqref{eq:dist-exp-moment-1} and the fact that $\rho$ is an eigenfunction of $\hPP$ to eigenvalue $1-\kappa$:
\begin{align}
&\hE_{x,y}e^{\lambda(\rho(X_T,Y_T) - \hPP^{T}\rho(x,y))}		\\
&= \hE_{x,y}\left(\hE_{X_{T-1},Y_{T-1}} \left[e^{\lambda(\rho(X_1,Y_1) - \hPP\rho(X_{T-1},Y_{T-1}))}\right]
e^{\lambda(\hPP\rho(X_{T-1},Y_{T-1}) - \lambda\hPP^{T}\rho(x,y)}\right)	\\
&\leq \left(1+ 2\sum_{n=2}^\infty\frac{|2\lambda|^n\sigma^{(n)}}{n!}\right)
\hE_{x,y}e^{\lambda(\hPP\rho(X_{T-1},Y_{T-1}) - \hPP^{T}\rho(x,y))}	\\
&= \left(1+ 2\sum_{n=2}^\infty\frac{|2\lambda|^n\sigma^{(n)}}{n!}\right)
\hE_{x,y}e^{\lambda(1-\kappa) (\rho(X_{T-1},Y_{T-1}) - \hPP^{T-1}\rho(x,y))}.
\end{align}
Iterating this estimate, we obtain
\begin{align}
&\hE_{x,y}e^{\lambda(\rho(X_T,Y_T) - \hPP^{T}\rho(x,y))}
\leq \prod_{t=0}^{T-1}\left(1+ 2\sum_{n=2}^\infty 
\frac{|2\lambda(1-\kappa)^t|^n\sigma^{(n)}}{n!}\right)	\\
&\leq \exp\left( 2\sum_{t=0}^{T-1}\sum_{n=2}^\infty 
\frac{|2\lambda(1-\kappa)^t|^n\sigma^{(n)}}{n!} \right) .	
\end{align}
The claim follows by summing $(1-\kappa)^{nt}$ over $t$.
\end{proof}

\begin{proof}[Proof of Corollary \ref{corollary:distance-concentration}]
Write $\beta=4J_\rho$ and $\alpha=2T$ or $\alpha=2(\kappa(2-\kappa))^{-1}$, depending on $\kappa=0$ or $\kappa>0$. By Theorem \ref{thm:distance-exponential-moment} and $\sigma_\rho^{(n)}\leq (2J_\rho)^n$
the conditions of Lemma \ref{lemma:concentration} are satisfied, and
\begin{align*}
\hP_{x,y}( \pm (\rho(X_T,Y_T)-(1-\kappa)^T\rho(x,y)) > r\beta ) \leq \exp\left( \frac{-\frac12 r^2}{\alpha+\frac13 r} \right).
\end{align*}
Consequently, 
\begin{align*}
\hP_{x,y}\left(\abs{\rho(X_T,Y_T)-(1-\kappa)^T\rho(x,y)} > r\beta \right) \leq 2 \exp\left( \frac{-\frac12 r^2}{\alpha+\frac13 r} \right),
\end{align*}
and the claim follows for $r'=4r$.
\end{proof}

\appendix

\section{Multi-valued maps and measurable selections}
In this section we look at a few basic facts about multi-maps or correspondences, limited to the context of this article. For a more general overview see \cite{ALIPRANTIS:BORDER:99}, chapters 16 and 17.
\begin{definition}
Let $F_1, F_2$ be Polish spaces. 
A multi-map or correspondence $\phi:F_1 \twoheadrightarrow F_2$ is a map from $F_1$ with values in the power set of $F_2$.

The upper inverse of $\phi$ of $F_1\subset F_2$ is given by
\begin{align}
\phi^u(A):=\{x\in F_1 : \phi(x)\subset A\}.
\end{align}

The correspondence $\phi$ is called weakly measurable if $\phi^u(A)$ is measurable for all closed $A\subset F_2$.
\end{definition}

\begin{theorem}[Kuratowski-Ryll-Nardzewski Selection Theorem]\label{thm:selection}
Let $F_1,F_2$ be Polish spaces and $\phi: F_1\twoheadrightarrow F_2$ be weakly measurable with non-empty closed vales. Then there exists $f:F_1\to F_2$ measurable with $f(x)\in \phi(x)$ for all $x\in F_2$.
\end{theorem}
See for example \cite{ALIPRANTIS:BORDER:99}, Theorem 17.13. for a reference.

\bibliography{BibCollection}{}
\bibliographystyle{plain}

\end{document}

%% file: Include.tex
\DeclareMathOperator{\supp}{supp}

\DeclareMathOperator{\diam}{diam}

\newcommand{\bE}{\ensuremath{\mathbb{E}}}

\newcommand{\bN}{\ensuremath{\mathbb{N}}}

\newcommand{\bP}{\ensuremath{\mathbb{P}}}

\newcommand{\bR}{\ensuremath{\mathbb{R}}}

\newcommand{\bZ}{\ensuremath{\mathbb{Z}}}

\newcommand{\ind}{\ensuremath{\mathbbm{1}}}

\newcommand{\cC}{\ensuremath{\mathcal{C}}}

\newcommand{\cF}{\ensuremath{\mathcal{F}}}

\newcommand{\cM}{\ensuremath{\mathcal{M}}}

\newcommand{\cP}{\ensuremath{\mathcal{P}}}

\newcommand{\cS}{\ensuremath{\mathcal{S}}}

\newcommand{\cW}{\ensuremath{\mathcal{W}}}

\newcommand{\abs}[1]{\left\vert \, #1 \, \right\vert}
\newcommand{\norm}[1]{\left\Vert \, #1 \, \right\Vert}

\newcommand{\normb}[1]{\interleave \, #1 \, \interleave}

\newcommand{\ddx}[1][1]{\ifnum#1=1 \frac{d}{dx} \else \frac{d^{#1}}{dx^{#1}} \fi}
\newcommand{\ddy}[1][1]{\ifnum#1=1 \frac{d}{dy} \else \frac{d^{#1}}{dy^{#1}} \fi}
\newcommand{\ddt}[1][1]{\ifnum#1=1 \frac{d}{dt} \else \frac{d^{#1}}{dt^{#1}} \fi}



%% file: Theorems.tex
\newtheorem{theorem}{Theorem}[section]
\newtheorem{lemma}[theorem]{Lemma}
\newtheorem{proposition}[theorem]{Proposition}
\newtheorem{corollary}[theorem]{Corollary}
\newtheorem{definition}[theorem]{Definition}

\newtheorem{example}[theorem]{Example}

\newtheorem{assumption}{Assumption}
\newenvironment{remark}[1][Remark]{\begin{trivlist}
\item[\hskip \labelsep {\bfseries #1}]}{\end{trivlist}}